\documentclass[12pt]{article}
\usepackage{amsmath,amssymb,amsfonts,amsthm}
\usepackage{eucal} 

\newcommand{\cF}{\mathcal{F}}
\newcommand{\C}{\mathbb{C}}
\newcommand{\Q}{\mathbb{Q}}
\newcommand{\Z}{\mathbb{Z}}
\newcommand{\cO}{\mathcal{O}}

\newcommand{\bW}{\mathsf{W}}
\newcommand{\cL}{\mathcal{L}}
\newcommand{\cI}{\mathcal{I}}
\newcommand{\cJ}{\mathcal{J}}
\newcommand{\cK}{\mathcal{K}}
\newcommand{\bE}{\mathsf{E}}
\newcommand{\bN}{\mathsf{N}}

\newcommand{\Pp}{\mathbf{P}^1}

\newcommand{\rang}{\right\rangle}

\newcommand{\lv}{\left |}

\newcommand{\Lt}{\mathfrak{t}}
\newcommand{\fZ}{\mathfrak{Z}}
\newcommand{\cZ}{\mathcal{Z}}
\newcommand{\fA}{\mathfrak{A}}
\newcommand{\lan}{\left\langle}
\newcommand{\ran}{\right\rangle}

\DeclareMathOperator{\End}{End}
\DeclareMathOperator{\Hilb}{Hilb}
\DeclareMathOperator{\Ext}{Ext}
\DeclareMathOperator{\rk}{rk}
\DeclareMathOperator{\tr}{tr}
\DeclareMathOperator{\str}{str}
\DeclareMathOperator{\Id}{Id}
\DeclareMathOperator{\pt}{pt}
\DeclareMathOperator{\ch}{ch}
\DeclareMathOperator{\td}{td}

\DeclareMathOperator{\Lie}{Lie}
\DeclareMathOperator{\supp}{supp}
\DeclareMathOperator{\Aut}{Aut}
\DeclareMathOperator{\Ker}{Ker}
\DeclareMathOperator{\zee}{\mathfrak{z}}
\DeclareMathOperator{\aut}{aut}

\newtheorem{Theorem}{Theorem}
\newtheorem{Lemma}{Lemma}
\newtheorem{Corollary}[Lemma]{Corollary}

\begin{document}
\title{Exts and Vertex Operators}
\author{Erik Carlsson and Andrei Okounkov}
\date{} \maketitle

\begin{abstract}
The direct product of two Hilbert schemes of the same surface has
natural K-theory classes given by the alternating Ext groups between 
the two ideal sheaves in question, twisted by a line bundle. 
We express the Chern classes of these
virtual bundles in terms of Nakajima operators. 
\end{abstract}
\textit{}
\section{Introduction}

\subsection{}

Let $S$ be a nonsingular quasi-projective surface. 
The Hilbert scheme $\Hilb_n S$ of $n$ points on $S$
has been the focus of numerous recent studies, see e.g.\ \cite{Goe,L,Nak2} for 
a survey. In particular, the cohomology of $\Hilb_n S$ has been 
described using certain operators acting on 
\begin{equation}
  \label{defcF}
   \cF = \bigoplus_n \cF_n\,, \quad \cF_n=H^*(\Hilb_n S,\Q)\,,
\end{equation}
introduced by Nakajima \cite{Nak1} and Grojnowski \cite{Groj}. 

In this paper we consider another natural set of operators, $\bW(\cL)$, depending
on a line bundle $\cL$ on $S$. These operators act on $\cF$, and are
defined in terms of Chern classes of sheaves of $\Ext$-groups. We prove 
an explicit formula for $\bW(\cL)$ in terms of the Nakajima 
operators, describing $\bW(\cL)$ as a \emph{vertex operator} acting on $\cF$. 

Similar operators may be defined and studied for more general moduli spaces 
of sheaves. In this paper we will be content with the case of rank-one torsion-free
sheaves on a surface (i.e.\ the Hilbert schemes of points), leaving 
generalizations to future papers. 

The main application we have in mind concerns the case of $T$-equivariant 
cohomology of $\Hilb_n \C^2$, with respect to the natural 
action of $T\cong (\C^*)^2$ on $\C^2$. In this case, the trace of $\bW(\cL)$
becomes one of Nekrasov's instanton partition functions, specifically the one 
with matter in the adjoint representation. Equivariant localization 
translates our formula into 
a rather surprising Pieri-type formula for Jack symmetric functions.

\subsection{}

Let $S$ be a nonsingular surface and $\cL$ a line bundle on $S$. 
The Hilbert scheme $\Hilb_n=\Hilb_n S$ parametrizes 
length $n$ subschemes $Z\subset S$ or, equivalently, ideal sheaves 
$$
I = \Ker \left(\cO_S \to \cO_Z\right)\,, 
$$
of colength $n$. It comes with a universal subscheme 
$$
\fZ \subset \Hilb_n S \times S
$$
whose fiber over $Z\in \Hilb_n$ is $Z$ itself. Consider
$$
\cZ_i = p_{i3}^*\left(\cO_{\fZ}\right) \in K(\Hilb_k S\times \Hilb_l S \times S)\,, \quad i=1,2\,, 
$$
where, for example, $p_{13}$ denotes the projection to the first and third factors. The main object of this 
paper is the following class
$$
\bE = {p_{12}}_*\left(\left(
\cZ_1^\vee+\cZ_2-\cZ_1^{\vee}\cdot\cZ_2\right)\cdot
p^*_{3}(\cL)\right) \in K(\Hilb_k \times \Hilb_l), 
$$
where checks denote $K$-theoretic duals. Note 
that the fibers of $p_3$ intersect 
the supports of $\cZ_i$ in finite sets,  
hence ${p_{12}}_*$ is well-defined 
without assuming $S$ is proper. 

If $S$ is proper, then 
\begin{equation}
\bE\big|_{(I_1,I_2)} = \chi(\cL) - \chi(I_1,I_2\otimes\cL)\,,
\label{Echi}
\end{equation}
where 
$$
\chi(F,G) = \sum_{i=0}^{2} (-1)^i \, \Ext^i(F,G) 
$$
for a pair of coherent sheaves $F$ and $G$ on $S$. The subtraction in 
\eqref{Echi} is also meaningful if $I_1$ and $I_2$ are invariant 
under the action of a torus $T\subset \Aut(S)$ such that $S^T$ is 
proper. In that case,  the $\Ext$-groups have finite dimensional $T$-eigenspaces, and their difference is finite-dimensional,
making \eqref{Echi} well-defined in $K_T(pt)$. For a reference on $K$-theory, especially the equivariant case, see \cite{CG}.
For reference on equivariant cohomology in algebraic geometry, see Fulton's notes \cite{F}.

\subsection{}

We have 
$$
\rk \bE = k+l
$$
and we denote 
$$
e(\bE) = c_{k+l}(\bE) \,.
$$
We define the operator $\bW(\cL) = \bW(\cL,z) \in \End(\cF)[[z,z^{-1}]]$ by its matrix elements
\begin{equation}
  \label{defW}
 \left(\bW(\cL)\, \eta, \xi\right) = z^{l-k}\int_{\Hilb_k \times \Hilb_l} p_1^*(\eta) \cup p_2^*(\xi)\cup e(\bE) \,,
 \end{equation}
where
$$
\eta \in H_T^*(\Hilb_k)\,,\quad \xi \in H_T^*(\Hilb_l),
$$
and
\[(\eta,\xi) = \int \eta \cup \xi.\]
Here and in what follows integration is defined  by equivariant localization:
\begin{equation}
  \int_X \eta = \sum_a \int_{F_a} \frac{\iota_a^* \eta}{e_a} \in \C(\Lt^*),\quad \Lt = \Lie(T),
\label{defInt}
\end{equation}
where 
$$
\iota_a: F_a \to X
$$
is the inclusion of a connected component of 
$X^T$  and $e_a$
is the Euler class of its normal bundle. The fixed-point loci 
will be always assumed proper, so $\int_{F_a}$ in 
\eqref{defInt} is the usual nonequivariant integral.  
The localization formula 
states that the integrals as defined in \eqref{defInt} commute with proper equivariant 
pushforwards, see \cite{GS} for a reference.
 
For each $v \in \cF$, $\bW(\cL)\cdot v$ is a formal power series in
$z$, $z^{-1}$ with coefficients in $\cF$ such that that coefficient of 
$z^{N}$ is $0$ for $N \ll 0$.

\subsection{}

Note that one can always twist $\cL$ by a torus character, even when the 
torus action on $S$ is trivial. The result is a scalar torus action on 
the fibers of $\bE$ and hence $e(\bE)$, as $T$-equivariant class, 
effectively includes all smaller Chern classes of $\bE$.  As we will see, 
\begin{equation}
\label{vanish}
c_i(\bE)=0\,, \quad i> k+l \,.
\end{equation}
Note also that by Serre duality,
\[\chi(I,J \otimes \cL) = \chi(J,I\otimes \cK\otimes \cL^{-1}),\] 
so that
$$
\bW(\cL,z)^* = (-1)^{\bN} \, \bW(\cK-\cL,z^{-1}) \, (-1)^{\bN},
$$
where $\cK$ is the canonical bundle of $S$, and $\bN$ is the number-of-points
operator,
$$
\bN\big|_{\cF_n} = n \, \Id\,.
$$

\subsection{}

Our goal is to relate the operator $\bW(\cL)$ to Nakajima operators on $\cF$
which are defined as follows. For $k < l$ define a cycle by
\[\fA = \left\{ (I,x,J)\ | J \subset I,\supp(J/I) = \{x\}\right\} \subset \Hilb_k S\times S \times \Hilb_l S.\]
For each $\gamma \in H^T_*(S)$ define an operator by
$$
(\alpha_{k-l}(\gamma) \, \eta, \xi ) =  \int p_1^*(\eta) \cup p_3^*(\xi) \cap p_2^*(\gamma) \,\cap [\fA]  , 
\quad \eta\in\cF_k\,, \xi\in \cF_l\, 
$$
and let $\alpha_n(\gamma) = \alpha_{-n}(\gamma)^*$ for $n>0$.
We will use the shorthand $\alpha_{-n}(\cL)$ for the case when 
$\gamma$ is the Poincar\'e dual of $c_1(\cL)$. 
Nakajima's theorem says that
\[\left[\alpha_i(\gamma),\alpha_j(\gamma')\right] = (-1)^{i-1}i\delta_{i,-j} (\gamma,\gamma')_S,\]
where the brackets denote
the supercommutator with respect to the standard $\Z/2$-grading in 
cohomology.

Now we can state our main result 

\begin{Theorem}\label{t_main} 
  \begin{equation}
    \label{mainT}
  \bW(\cL) = 
\exp\left(\sum_{n>0} \frac{(-1)^{n-1}z^{n}}{n} \, \alpha_{-n}(\cL)\right)
\exp\left(-\sum_{n>0} \frac{z^{-n}}{n} \, \alpha_{n}(\cK-\cL)\right) \,.
\end{equation}
\end{Theorem}

In particular, one can take $k=0$ in \eqref{defW}. The second factor 
in \eqref{mainT} fixes the vacuum vector 
$$
 \lv \rang = 1 \in H^*(\Hilb_0) \subset \cF\,, 
$$
while $\bE$ becomes the tautological bundle $\chi((\mathcal{O}/J)\otimes \cL)$. In this case, \eqref{mainT} 
specializes to 
a formula of M.~Lehn, see corollary 6.6 in \cite{L}. 

The product in \eqref{mainT} is a \emph{vertex operator} in the sense that it may 
be uniquely characterized by its commutation relations with the Heisenberg algebra
spanned by the Nakajima operators. It would be very interesting to find a geometric 
reason for this. Various connections and parallels between Hilbert schemes and 
vertex operator algebras may be found in the work of Li, Qin, and Wang, see e.g.\ 
\cite{LQW}. 

\begin{Corollary}\label{c_str}
$$
\str q^{\bN} \, W(\cL) = \prod_n (1-q^n)^{(\cL,\cK-\cL)-e(S)} \,,
$$
where the right hand side is the Taylor series about $q=0$, and $e(S)$ is the Euler characteristic of $S$. 
\end{Corollary}

\noindent
This will be deduced from the main theorem in Section \ref{s_pr_cor}

\subsection{}

Parallel structures exist in the 
$K$-theory of Hilbert schemes, see \cite{CNO}.
The authors would like to thank Davesh Maulik, Nikita Nekrasov, and the referees.

\section{Applications} 

\subsection{Nekrasov theory}

\subsubsection{}

If $\cL=\cO$ then 
$$
\bE \Big|_{\textup{diagonal in }\Hilb_n \times \Hilb_n} = T \Hilb_n\,,
$$
in $K$-theory, see Proposition 2.2 in \cite{EGL}. This explains why 
$$
\str q^{\bN} \, W(\cO) = \prod_n (1-q^n)^{-e(S)} \,,
$$
agrees with the generating functions for $e(\Hilb_n)$ obtained from 
G\"ottsche's formula \cite{Goe}. 

Twisting $\cO$ by a torus character $m$, we get the Chern polynomial of 
the tangent bundle. Hence 
$$
\str q^{\bN} \, W(\cO(m))  = \sum_n q^n \int_{\Hilb_n} c(T \Hilb_n,m) 
$$
By definition, this series is the instanton part of Nekrasov partition 
function with for the $U(1)$-adjoint matter of mass $m$, 
a subject of considerable 
physical interest,  see for example \cite{NO}. In particular,  
corollary \ref{c_str} generalizes formula (6.12) in \cite{NO}. 
More generally, we have the following 

\begin{Corollary}\label{c_tr} 
For any $\phi \in \cF$,  
\begin{equation}
  \sum_{n} q^n 
\int_{\Hilb_n} \phi \cup c(T \Hilb_n,m) = \str q^{\bN} \, \Phi \, 
W(\cO(m))\,, \label{corrF}
\end{equation}
where $\Phi: \cF \to \cF$ is the operator of multiplication by $\phi$ 
in the cohomology of $\Hilb_n$.  
\end{Corollary}

\subsubsection{}

Series of the form \eqref{corrF} are \emph{correlation functions}
in Nekrasov theory. They are studied in detail in \cite{Car}. 
Generalizing \eqref{corrF}, we denote 
$$
\lan\phi\ran = \str q^{\bN} \, \Phi \, W(\cL)= \sum_{n} q^n 
\int_{\Hilb_n S} \phi \cup e(\bE)\, \Big|_{\textup{diagonal}} \,. 
$$
Corollary \ref{c_str} is the computation of $\lan 1 \ran$. The 
next simplest correlator is computed in the following 

\begin{Corollary}\label{c_c1} 
Let $\phi=c_1(\cO/\cJ)$ be the $1$st Chern class of the 
tautological bundle $\cO/\cJ$ on $\Hilb_n S$. Then
\begin{equation}
\frac{\lan c_1(\cO/\cJ) \ran}{\lan 1 \ran} = 
\tfrac12\left(E_2(q)-E_3(q)\right) \, \int_S (c_1 c_2 - c_3) (TS \oplus \cL)
\,.
\label{<c1>}
\end{equation}
\end{Corollary}

\noindent 
Here the series
$$
E_k(q) = \sum_{n>0} n^{k-1} \frac{q^n}{1-q^n}
$$
is a $q$-analog of $\zeta(k)$ and also an $SL_2(\Z)$-Eisenstein 
series without 
constant term for $k$ positive even. 

The characteristic class $c_1 c_2-c_3$
in \eqref{<c1>} refers to the rank $3$ bundle $TS \oplus \cL$ on the 
surface $S$. Its integral over $S$ has degree $1$ in equivariant 
parameters.

\subsubsection{Chern classes of $T\Hilb$} 

Corollary \ref{c_tr} gives a procedure to compute the Chern 
classes of the tangent bundle to the Hilbert scheme $\Hilb_n S$ 
in any basis $\{\phi_i\}$ of the cohomology, such as the 
Nakajima basis. To use it, one has to identify the
corresponding operators $\{\Phi_i\}$ of multiplication by 
$\phi_i$ in cohomology. 

One algorithm for producing the multiplication operators $\Phi_i$ 
is discussed in \cite{L}. There are also other approaches, see
for example \cite{LQW,MO}. 

\subsection{Proof of corollaries \ref{c_str} and \ref{c_c1}}
\label{s_pr_cor}

\subsubsection{}

Write Theorem \ref{t_main}
as $W(\cL) = \Gamma_-(-\cL,-z)\Gamma_+(\cL-\cK,z)$, where
\[\Gamma_{\pm}(\cL,z) = \exp\left(\sum_{n>0} \frac{z^{\mp n} \alpha_{\pm n}(\cL)}{n}\right).\]
Using Nakajima's commutation relations,
\begin{align*}
&\str q^{\bN}\, \Gamma_-(-\cL,-z) \Gamma_+(\cL-\cK,z) \\
&\qquad =\str   q^{\bN}\, \Gamma_+(\cL-\cK,z)\Gamma_-(-\cL,-zq) \\
&\qquad =(1-q)^{(\cL,\cK-\cL)} \,\str   q^{\bN}\, \Gamma_-(-\cL,-zq)\Gamma_+(\cL-\cK,z) \\
&\qquad =\prod_n(1-q^n)^{(\cL,\cK-\cL)}\, \str q^{\bN}\,\Gamma_+(\cL-\cK,z) \,. 
\end{align*}
The trace of these operators is defined as a formal power series in $q$,
which makes sense because each coefficient is an operator whose trace is a finite sum.
We have used the rule that $\str(AB) = \str(BA)$ whenever $A$ and $B$ are operators which raise the degree by an even amount.
Now corollary \ref{c_str} follows from G\"ottsche's formula, and the fact that $\Gamma_+$ is lower-triangular.

\subsubsection{}

By the reduction procedure used in the proof of Theorem \ref{t_main}, 
it suffices to prove \eqref{<c1>} for a toric surface $S$. By 
localization, we can further reduce to $\C^2$. As in Section 
\ref{s_loc_C2}, we denote by $t_1,t_2$ the equivariant weights 
of $T \C^2$ and by $m$ the equivariant weight of $\cL = \cO_{\C^2}$. Then 
$$
 \int_{\C^2} (c_1 c_2 - c_3) (T\C^2 \oplus \cL)  = 
\frac{(t_1+t_2)(t_1+m)(t_2+m)}{t_1 t_2} 
$$
On the other hand, the operator $\Phi$ for $\phi=c_1(\cO/\cJ)$
is given by a formula due to M.~Lehn \cite{L1} 
$$
\Phi = -\frac{t_1+t_2}{2}\sum_{k\geq 1} (k-1)\alpha_{-k}\alpha_k+
\sum_{k,l \geq 1} \frac{t_1t_2}{2} \alpha_{-k}\alpha_{-l}\alpha_{k+l} -\frac{1}{2}\alpha_{-k-l}\alpha_{k}\alpha_{l}\,,
$$
in which we use the shorthand
\[\alpha_{-n} = \alpha_{-n}(1),\quad \alpha_n = (-1)^{n-1}t_1t_2\alpha_{-n}^*\]
so that $[\alpha_m,\alpha_n] = m\delta_{m,-n}$. Compare Lehn's commutation relations with the commutator of the above formula for $\Phi$ with $\alpha_n$.

We use the commutation relations
\[\left[\alpha_{\pm k}, \Gamma_{\mp}^m(z) \right]=  \pm mz^{\mp k} \Gamma_{\mp}^m(z),\quad \Gamma_{\pm}^m(z) = \exp\left(m\sum_k \frac{z^{\mp k}}{k}\alpha_{\pm k}\right)\]
repeatedly to 
move the vertex operators around the trace as in 
the proof of corollary \ref{c_str} above. 
It is slightly simpler to use the dual of the vertex operator under the inner-product \eqref{jsp}, because it eliminates some awkward signs. Since 
\[c_1(\cK) = -t_1-t_2 \in \C[t_1,t_2], \quad c_1(\cL)=m,\]
we get
\[W(m)^* = \Gamma_-^{m+t_1+t_2}(1)\Gamma_+^{m/t_1/t_2}(1).\]
Taking the trace against one of the terms in $\Phi$ gives, for instance,
\[ \tr q^N \alpha_{-k-l}\alpha_k\alpha_l \Gamma_-^{m+t_1+t_2}(1)\Gamma_+^{m/t_1/t_2}(1) = \]
\[ \frac{m(m+t_1+t_2)^2}{t_1t_2}\frac{q^{k+l}}{(1-q^k)(1-q^l)(1-q^{k+l})} \lan 1 \ran.\]
Combining all three terms in $\Phi$, we get
\begin{multline}\label{3series} 
-\frac{\lan c_1(\cO/\cI) \ran}{\lan 1 \ran} 
= \frac{t_1+t_2}{2}\sum_k k(k-1)\frac{q^k}{1-q^k}+\\
\frac{(t_1+t_2)(m+t_1+t_2)m}{2t_1t_2} \sum_{k} (k-1)\frac{q^k}{(1-q^k)^2}+\\
\frac{(t_1+t_2)(m+t_1+t_2)m}{2t_1t_2}\sum_{k,l}\frac{q^{k+l}}{(1-q^k)(1-q^l)(1-q^{k+l})}.
\end{multline}
The first two $q$-series above are easily seen to equal 
$E_3-E_2$ and $q\frac{d}{dq} E_1 - E_2$, respectively. The 
third is identified in the following 

\begin{Lemma} 
  \begin{equation}
\sum_{k,l>0} \frac{q^{k+l}}{(1-q^k)(1-q^l)(1-q^{k+l})} = 
E_3 - q \frac{d}{dq} E_1  \,. 
\label{ident}
\end{equation}
\end{Lemma}

\begin{proof}
We denote $[n] = q^{n/2}-q^{-n/2}$ and consider
$$
f(z) = \sum_{n>0} \frac{z^n+z^{-n}}{[n]} \,, \quad 
|q^{1/2}| < |z| < |q^{-1/2}| \,. 
$$
We have 
\begin{equation}
\textup{LHS} =  - \frac16 \, [z^0] \,f(z)^3
\label{LHS6}
\end{equation}
where $[z^0]$ denotes the coefficient of $z^0$. {}From 
$$
\frac{1}{[n][n+l]} = \frac{1}{[l]} \left(\frac{q^{n/2}}{[n]} - 
\frac{q^{(n+l)/2}}{[n+l]} \right)
$$
we find, for $l\ne 0$ and $|q|<1$, 
the following series telescopes:
$$
[z^l] \,f(z)^2 = 
\sum_{a+b=l,\, ab>0 \,\,} \frac{1}{[a][b]} +
\frac{2}{[l]} \sum_{a=1}^{|l|} \frac{q^{a/2}}{[a]} \,. 
$$
Comparing this to \eqref{LHS6} and setting $(n_1,n_2)=(l,a)$, we find 
$$
\textup{LHS} = - \sum_{n_1 \ge n_2 >0 } \frac{q^{n_2/2}}{[n_2]\, [n_1]^2} \,. 
$$
Now \eqref{ident} follows at once from the following 
analog
\begin{equation}
\sum_{n_1 \ge n_2 >0 } \frac{q^{-n_2/2}}{[n_2]\, [n_1]^2}
= \sum_{n>0} \frac{q^{n/2}}{[n]^3}
\label{zeta21}
\end{equation}
 of Euler's $\zeta(2,1)=\zeta(3)$ identity
proven by Bradley. Namely, \eqref{zeta21}
 is the $k=n=2$ case of corollary 4 in 
\cite{Bradley}. 
\end{proof}

One collects the terms in \eqref{3series} to finish the 
proof of corollary \ref{c_c1}.

\subsection{Pieri-type formulas for Jack polynomials}

Recall that algebra $\Lambda$ of symmetric functions over $\Q$ is freely generated by the power-sum symmetric functions $p_k$,
where $k=1,2,\dots$, that is,  
$$
\Lambda = \Q[p_1,p_2,p_3,\dots] \,.
$$
Consider the Jack inner product $(\,\cdot\,,\,\cdot\,)_\theta$ 
on $\Lambda$ with parameter $\theta$. This is the unique inner product such that  
$$
p_k^* = \frac{k}{\theta} \, \frac{\partial}{\partial p_k}
$$
where $p_k$ is viewed as the corresponding multiplication operator. We have $\theta=1/\alpha$, where $\alpha$ is 
the parameter used in Macdonald's book \cite{Mac}. 
The integral form $J_\mu$ of Jack symmetric functions is obtained by Gram-Schmidt orthogonalization of monomial 
symmetric functions with respect to $(\,\cdot\,,\,\cdot\,)_\theta$. It is normalized so that the coefficient of 
$p_1^{|\mu|}$ in $J_\mu$ equals $1$. 

Let $E$ be the operator of multiplication by the sum of all elementary symmetric functions
$$
E = 1 + e_1 + e_2 + \dots = \exp\left(\sum_n \frac{(-1)^n}{n} \, p_n \right) \,.
$$
and $E^*$ the adjoint operator. Specialized to the equivariant cohomology of $\Hilb_n \C^2$, 
formula \eqref{mainT} may be restated as the following Pieri-type formula for Jack polynomials: 

\begin{Corollary} 
\begin{multline}
\left(E^m (E^*)^{\theta-m-1} J_\lambda, J_\mu \right)_\theta =  (-1)^{|\lambda|} \, \theta^{-|\lambda|-|\mu|} \times \\
\prod_{\square\in \lambda} (m+a_\lambda(\square)+1+\theta \, l_\mu(\square)) 
\prod_{\square\in \mu} (m-a_\mu(\square)-\theta (l_\lambda(\square)+1)) \,.
\label{evalJack} 
\end{multline}
\end{Corollary}

\noindent
We note that in the product \eqref{evalJack} the ranges $\square\in \lambda$ and $\square\in \mu$ may be 
interchanged, see the proof of Lemma \ref{Lhooks} below.

\section{Proof of Theorem \ref{t_main}}

\subsection{}
Let $\Gamma(\cL)$ denote the right-hand side of \eqref{mainT} so 
\begin{equation}
\bW(\cL) = \Gamma(\cL)
\label{mainT2}
\end{equation}
is what we need to prove. As a first step, we reduce the general formula \eqref{mainT2} to the special case of 
the equivariant cohomology of a toric surface $S$. 

Consider the universal ideal sheaf $\cI$ on $\Hilb_k S \times S$ 
and the K\"unneth components of its Chern character
$$
\sigma_i(\gamma) = \int_S \ch_{i+2}(\cI)\, \gamma  
\quad \in H^{2i+\deg \gamma}(\Hilb_k) \,.
$$
It is known that these classes generate $H^*(\Hilb_k)$ as a ring, 
that is, the monomials in these classes span $H^*(\Hilb_k)$, 
see e.g.\ \cite{LQW} and Section 6 in \cite{L}.  

For given $k$ and $l$, consider a matrix coefficient of $\bW(\cL)$ between 
two such monomials 
\begin{equation}
\label{matrix_coeff}
\left(\bW(\cL) \prod \sigma_{p_i}(\eta_i), \prod \sigma_{q_j} (\xi_j)\right) \,.
\end{equation}
We may assume that the cohomology classes $\eta_i$ and $\xi_j$ have a well-defined parity. 

Using the Grothendieck-Riemann-Roch formula, we may express \eqref{matrix_coeff} 
as an integral over 
\begin{equation}
\Hilb_k S\times \Hilb_l S \times S  \times S \times \dots \times S 
\label{prodS} 
\end{equation}
of a universal expression in characteristic classes of 
$$
\pi^*_{rs} \, \cI\,, \quad r=1,2 \quad s\ge 3 \dots\,,
$$
where $\pi_{rs}$ is the projection on the $r$th and $s$th factor in \eqref{prodS}, times
a class $\gamma_s$ of the form 
$$
\gamma_s\in\{\eta_i,\xi_j,\ch_k(\cL)\td_k(S)\} 
$$
for each of the $S$-factors in \eqref{prodS}. 

The induction scheme of Ellingsrud, G\"ottsche, and Lehn, see Section 3 in \cite{EGL}, 
gives a universal evaluation of any such integral. The result has the form 
\begin{equation}
\label{universal}
\sum_{\{3,4,\dots\}=\sqcup G_i} \pm \prod_i \int_S C_{G,i} \prod_{s\in G_i} \gamma_s\,,
\end{equation}
where the outer sum is over all partitions of the $S$-factors in \eqref{prodS} into disjoint 
groups $G_i$, the sign accounts for a possibly different ordering of odd cohomology classes 
in the original and final integral, and each $C_{G,i}\in H^*(S)$ is a universal polynomial, independent of the $\gamma_i s$,
applied to the characteristic classes of the tangent bundle of $S$. 

Next, we show that the matrix coefficients of $\Gamma(\cL)$ are given by a similar universal formula.
To do this, we use the calculus explained in Section 6 of \cite{L}, and subsequent generalizations in 
\cite{LQW} to express the operator of the cup product with $\sigma_{p}(\eta)$ in terms of suitable Nakajima operators.
We may then express the classes $\prod \sigma_{p_i}(\eta_i)$ in terms of operators of the form
\[\alpha_\lambda(\Delta_{\ell*} \gamma) = \sum_j \alpha_{\lambda_1}(\gamma(j,1))\cdots \alpha_{\lambda_\ell}(\gamma(j,\ell)),\]
applied to the vacuum. Here $\lambda$ is a generalized partition,
$\ell$ is the length of $\lambda$, $\Delta_\ell$ is the diagonal embedding,
$\gamma$ is the product of $\eta_i$ with a characteristic class of $S$,
and $\gamma(j,i)$ are the K\"unneth components of the diagonal,
$\Delta_{\ell*} \gamma = \sum_j \gamma(j,1)\otimes \cdots \otimes \gamma(j,\ell)$.

Expanding $\Gamma(\cL)$, 
it suffices to show that
\[\langle |\alpha_{\lambda^{(1)}}(\Delta_{\ell_1*} \gamma_1)\cdots 
\alpha_{\lambda^{(k)}}(\Delta_{\ell_k*} \gamma_k)| \rangle\]
is given by an expression of the form \eqref{universal} with $\gamma_s \in\{\gamma_i\}$. Expanding this expression,
and using the Heisenberg relations gives
\begin{multline*}
\left\langle \left|\alpha_{\lambda^{(1)}}(\Delta_{\ell_1*} \gamma_1)\cdots 
\alpha_{\lambda^{(k)}}(\Delta_{\ell_k*} \gamma_k)\right|\right\rangle = \\
\sum_{j_1,...,j_k} \sum_{X}
C(X) 
\prod_{(x,y)\in X} \int_S \gamma_{x_1}(j_{x_1},x_2)\cdot \gamma_{y_1}(j_{y_1},y_2) = \\
\sum_{X} C(X)  
\int_{S \times \cdots \times S} \gamma_1 \otimes \cdots \otimes \gamma_k 
\cdot f_X(p_{ij}^* \cO_\Delta)\,,   
\end{multline*}
%
%
where $X$ is a perfect matching of the set
\[\{(a,b),\, 1\leq a \leq k,\, 1 \leq b \leq l_a\},\]
$C(X)$ are numbers, and $f_X$ is some polynomial depending on $X$.
This is of the desired form.

We now show that it suffices to check the equality of $W(\cL)$ and $\Gamma(\cL)$ when $S$ is a toric variety.
First notice that grouping the $\ch_k(\cL)\td_k(S)$ with the universal polynomials $C_{G,i}$
gives an alternative expression of the same form as \eqref{universal}, but where $\gamma_s$ varies over the $\{\eta_i,\xi_j\}$,
$C_{G,i}$ is a universal polynomial in $\ch_k(\cL)$ and the characteristic classes of $S$, and $G$ partitions the set $\{1,...,d\}$,
where $d$ is the number of classes $\eta_i,\xi_j$.

Now suppose $W(\cL)$ agrees with $\Gamma(\cL)$ whenever $S$ is toric, but not for some nontoric $S$. Then
there exists a universal family of polynomials $C_{G,i}$ as in the previous paragraph, such
that \eqref{universal} vanishes as an equivariant integral whenever $S$ is a toric variety, but not on some nontoric variety.
Let $C$ be such a collection, and let $G$ be a partition such that $C_{G,i}$ is not
the zero polynomial for any $i$, and no proper refinement of $G$ also has this property. If $S$ is
a toric variety with at least $d$ fixed points $p_s$, let
\[\gamma_s = \sum_{t \in G_i} \left[p_t\right]\]
where $i$ is such that $s \in G_i$, and $[p_s]$ is the class of the fixed point $p_s$. 
These classes have the property that $\prod_s \gamma_s$ is nonzero
when $s$ varies within some $G_i$, but zero otherwise. It is simple to find toric $S,\cL$, such that
\[\prod_i \int_S C_{G,i} \prod_{s \in G_i} \gamma_s\]
is nonzero. But by the choice of $G$ and $\gamma_s$, this is the only contributing term in \eqref{universal}, which is assumed to vanish.

\subsection{}

Let $T_0 = U(1) \times \cdots \times U(1)\subset T$ be the compact torus with Lie algebra $\Lt_0$.
We have $S^{T_0}=S^{T}$ and since we now assume $S$ toric this is a finite set. 
Let $\{U_s\}$ be disjoint $T_0$-invariant open neighborhoods of the fixed points $s\in S$.
By equivariant localization and the obvious inclusion map
\begin{equation}
\bigsqcup_s \Hilb_{n_s} U_s \hookrightarrow \Hilb_n S,\quad n = \sum_s n_s,
\end{equation}
we obtain an isomorphism
\begin{equation}
  \cF_{S} \otimes_{\C[\Lt^*]}\C(\Lt^*) = 
\bigotimes_{s\in S^T} \cF_{U_s} \otimes_{\C[\Lt^*]}\C(\Lt^*)\label{factorF}
\end{equation}
Here each factor 
$$
\cF_{U_s} \cong \cF_{T_s S} 
$$
corresponds to the Hilbert scheme of points of $\C^2 \cong T_s S$ with 
the inherited torus action. 

In \eqref{factorF}, we have an obvious factorization
$$
\bW_S(\cL) = \bigotimes_{s\in S^T} \bW_{T_s S}\left(\cL\big|_{s}\right) \, .
$$
Furthermore, $\Gamma(\cL)$ factors similarly, since
\[j^*([\fA(S)]) = [\fA(U)]\, ,\]
where $j$ is the induced from the inclusion of an open neighborhood $U \subset S$, and $j^*$ is the 
restriction map on Borel-Moore homology. 
See \cite{B} or \cite{E} for a reference on equivariant Borel-Moore cohomology and intersection theory.
The corresponding decomposition
$$
\alpha_{-n}(\gamma) = \bigoplus_{s} \alpha_{-n}\left(\gamma\big|_{s}\right)\, ,
$$
leads to the desired factorization of $\Gamma(\cL)$,
which reduces our theorem to the case $S=\C^2$. 
For $S=\C^2$, the only line bundles are torus characters. 

\subsection{}

The Fock space $\cF_{\C^2}$ may be identified with symmetric functions in such a way that
Nakajima operators become the operators of multiplication by the power-sum symmetric 
functions, while the classes of torus-fixed points are mapped to properly normalized
Jack symmetric functions. Since the character of the torus action in the fiber of $\bE$
over any fixed point is easily determined (see below), the statement of the theorem 
may be rephrased purely as a statement about symmetric functions. 

Instead of attacking the symmetric function problem directly, we will first use 
geometric arguments to reduce it further to a simple special case.

\subsection{}

Consider the following setup: 
$$
S=\Pp \times \Pp\,,\quad T=\C^*\times \C^*\times \C^*\,,\quad \cL = \cO(m) \,,
$$
where $(z_1,z_2,z_3)\cdot (x,y) = (z_1x,z_2y)$, and $z_3 = e^m$ acts by scaling $\cL$.

Take two partitions $\mu$ and $\nu$ such that $|\mu|=k$ and $|\nu|=l$ and consider
\begin{equation}
w_{\mu,\nu} = w_S(\mu,\nu) = \left(\bW(\cL,1) \prod \alpha_{-\mu_i}(L_1)  \lv \rang, \prod \alpha_{-\nu_i}(L_2)  \lv \rang\right) \in \Z\,,
\label{wmn}
\end{equation}
where 
$$
L_1 =\{\pt\}\times\Pp\,, \quad L_2 = \Pp \times \{\pt\} \,. 
$$
Let also $w^{[ab]}_{\C^2}(\mu,\nu)$, $a,b\in \{0,\infty\}$, be the same expression over 
the Hilbert scheme of the chart of $\Pp \times \Pp$ by $\C^2$ that contains the point $(a,b)$.
Use the intersection of $L_1, L_2$ with $\C^2$ in place of $L_1, L_2$.

Since $w_{\mu,\nu}$ is an integral of a top-dimensional class over a complete space, it is independent 
of the equivariant parameters (including $m$) as well as the choice of 
the equivariant lifts of the classes $L_1$ and $L_2$. 
Similarly, the number
$$
g_{\mu,\nu} = \left(\Gamma(\cL,1) \prod \alpha_{-\mu_i}(L_1)  \lv \rang, \prod \alpha_{-\nu_i}(L_2)  \lv \rang\right) 
$$
may be computed in either equivariant or ordinary cohomology and, in particular, is independent of equivariant parameters. 
We can use this independence as follows. 

\subsection{}

Let us compute $w_{\mu,\nu}$ using equivariant localization. 
Each occurrence of the line $L_1$ may be lifted to two different classes in equivariant cohomology, namely, 
$$
L_1^0 =\{0\}\times\Pp\quad \textup{or} \quad L_1^\infty =\{\infty\}\times\Pp\,,
$$
and similarly for $L_2$. Such a lifting corresponds to a decomposition
\[\mu = \mu^{[0]}\sqcup \mu^{[\infty]},\quad\nu = \nu^{[0]}\sqcup \nu^{[\infty]}.\]
Any choice of lifting allows us to compute $w$, $g$ by localization, so
the answer must be independent of this lifting.
In particular, using the above factorization of $\bW(\cL)$,
\begin{equation}
\label{winduction}
w_{\mu,\nu} = \sum_{\mu^{[00]},\mu^{[0\infty]},\mu^{[\infty0]},\mu^{[\infty\infty]}}\quad
\sum_{\nu^{[00]},\nu^{[0\infty]},\nu^{[\infty0]},\nu^{[\infty\infty]}}
\prod_{a,b} w^{[ab]}_{\C^2}(\mu^{[ab]},\nu^{[ab]}),
\end{equation}
where $a,b \in \{0,\infty\}$, and
the sum is over terms such that
$\mu^{[a0]} \sqcup \mu^{[a\infty]} = \mu^{[a]}$, 
$\nu^{[0b]} \sqcup \nu^{[\infty b]} = \nu^{[b]}$,
for any fixed $\mu^{[0]}$,$\mu^{[\infty]}$,$\nu^{[0]}$,$\nu^{[\infty]}$.
Using the parallel factorization for $\Gamma$, \eqref{winduction}
also holds when $\bW$ is replaced with $\Gamma$, and $w$ is replaced with $g$.

Equation \eqref{winduction} is a collection of relations for each decomposition of $\mu$ and $\nu$.
These formulas for both $w$ and $g$ give an induction step for proving
that $w_{\C^2} = g_{\C^2}$, which is sufficient to prove
the theorem. First, setting $\mu^{[a]} = \mu$, $\nu^{[b]} = \nu$, and solving for $w^{[ab]}_{\C^2}$,
we get
\[w^{[ab]}_{\C^2}(\mu,\nu) = w_S(\mu,\nu)-...\]
where the remaining part of the sum consists of lower order terms. In other words, they involve
$w^{[cd]}_{\C^2}(\mu',\nu')$ with $\ell(\mu') \leq \ell(\mu)$, $\ell(\nu)' \leq \ell(\nu)$, and either
$\ell(\mu') < \ell(\mu)$ or $\ell(\nu') < \ell(\nu)$. If $\ell(\mu) > 1$ or $\ell(\nu) > 1$, then using
any decomposition of $\mu$ or $\nu$ into strictly smaller components yields $w_S(\mu,\nu)$ as a 
function of the lower order terms $w_{\C^2}^{[cd]}(\mu',\nu')$.
Substituting this in, we arrive at a formula for $w_{\C^2}^{[ab]}(\mu,\nu)$ in
lower order terms which also holds for $g_{\C^2}(\mu,\nu)$ whenever $\ell(\mu) > 1$ or $\ell(\nu) > 1$.
The base cases $\mu = \emptyset$, $\nu = \emptyset$ are covered by Lehn's theorem, 
so what is left is to prove the case $\ell(\mu) = \ell(\nu) = 1$.

\subsection{}\label{s_single_product}

In fact, by the induction step above, we only need to check that 
$$
\left(\bW(\cL) \, \eta, \xi\right) = \left(\Gamma(\cL) \, \eta, \xi\right)
$$
for an arbitrary pair of classes 
$$
\xi \in H_T^*(\Hilb_k \C^2)\,, \quad \eta \in H_T^*(\Hilb_l\C^2)\,, 
$$
having nonzero inner product with the classes $\alpha_{-k} \lv\rang $ and $\alpha_{-l} \lv\rang $, respectively. 

To do this computation, we first need to set up the equivariant localization. 

\subsection{}\label{s_loc_C2} 

Let $T\cong (\C^*)^2 \subset GL(2)$ be the standard maximal torus, i.e.\
$$
\Lie T = \left\{ 
\begin{pmatrix}
 t_1 & \\
& t_2 
\end{pmatrix}
\right\}
$$
In equivariant cohomology, we may identify $\cF_{\C^2}$ with symmetric polynomials over $\Q[t_1,t_2]$ by 
requiring that $\lv\rang$ corresponds to $1$ and 
$$
\alpha_{-n}(1) \mapsto p_n\,,
$$
that is, the Nakajima operators become operators of multiplication by power-sum functions. It is known, see 
\cite{LQW_Jack,Vass} and also e.g.\ \cite{OP}, that the classes of fixed points then correspond to properly 
normalized Jack polynomials. 

Concretely, take a monomial ideal 
$$
I_\lambda = (x_1^{\lambda_i} x_2^{i-1}) \subset \C[x_1,x_2] \,.
$$
Then 
$$
[I_\lambda] \mapsto t_2^{|\lambda|} \, J_\lambda \Big |_{p_i = t_1 p_i} \,, 
$$
where $[I_\lambda]$ is the cohomological dual to the fixed point $I_\lambda \in \Hilb_n \mathbb{C}^2$, 
and $J_\lambda$ is the integral Jack polynomial with parameter
$$
\theta = - t_2/t_1 \,.
$$
The corresponding inner-product is
\begin{equation}
\label{jsp}
\left(p_\lambda,p_\mu\right)_{t_1,t_2} = \delta_{\lambda,\mu}\frac{(-1)^{|\lambda|-\ell(\lambda)}}{\zee(\lambda)t_1^{\ell(\lambda)}t_2^{\ell(\lambda)}},
\end{equation}
where
\[\zee(\lambda) = \frac{1}{\aut(\lambda)\prod_i \lambda_i}\]
is the usual factor \cite{Mac}.

\subsection{}
Our next goal is to compute the character of the $T$-module 
$$
\bE\Big|_{(I_\lambda,I_\mu)} = \chi(\cO,\cO) - \chi(I_\lambda,I_\mu) \,.
$$
Given a $T$-module $V$ with finite-dimensional weight spaces, 
we denote by $\left[ V\right]$ denote its character in $\C[[z_1,z_1^{-1},z_2,z_2^{-1}]]$. 
Concretely, $[V]$  is represented
by the trace of element $(z_1,z_2)\in T$ in its action on $V$. 

\begin{Lemma}\label{Lhooks} 
\begin{equation}
 \label{sum_hooks}
\left[\bE\Big|_{(I_\lambda,I_\mu)}\right] = 
\sum_{\square\in \lambda} z_1^{a_\lambda(\square)+1} \, z_2^{-l_\mu(\square)}+
\sum_{\square\in \mu} z_1^{-a_\mu(\square)}\, z_2^{l_\lambda(\square)+1}
 \,. 
\end{equation}
\end{Lemma}

\begin{proof}

By Riemann-Roch, we have 
$$
\chi(I_\lambda,I_\mu)  = \frac{[I_\lambda^{\vee}]\, [I_\mu] }{[\cO^{\vee}]}=\frac{{[I_\lambda]^\vee}\, [I_\mu] }{[\cO]^{\vee}},
$$
where ${f(z_1,z_2)}^{\vee} = f(z_1^{-1},z_2^{-1})$. Substituting
$$
[I_\mu] = \sum_{i\ge 1} \frac{z_1^{-\mu_i} \, z_2^{1-i}}{1-z_1^{-1}}\,, 
\quad 
[I_\lambda] = \sum_{j\ge 1} \frac{z_2^{-\lambda'_j} \, z_1^{1-j}}{1-z_2^{-1}}\,, 
$$
where $\lambda'$ is the transposed diagram, into the above formula yields
$$
\left[ \bE\Big|_{(I_\lambda,I_\mu)}\right] = \sum_{i,j\ge 1} z_1^{j-\mu_i} \, z_2^{\lambda'_j-i+1}  - \sum_{i,j\ge 1} z_1^{j} \, z_2^{1-i} \,.
$$

Now observe the terms for which the exponent of $z_1$ is $\le 0$ occur only in the first sum and correspond to the second 
sum in \eqref{sum_hooks}. For these terms then, the two formulas are equal.
The remaining terms may be determined using Serre duality which implies
$$
\left[ \bE\Big|_{(I_\lambda,I_\mu)} \right]  =  z_1 z_2 \left[  \bE\Big|_{(I_\mu,I_\lambda)}\right]^\vee \,.
$$
Note that the same argument with the roles of $z_1$ and $z_2$ interchanged yields the same formula 
except the ranges of summation $\square\in \lambda$ and $\square\in \mu$ are interchanged. 
\end{proof}

The Lemma shows that \eqref{mainT} for the equivariant cohomology of 
$\Hilb_n\C^2$ is equivalent to the Pieri-type 
formula \eqref{evalJack}.  

We also note that the character is a sum of exactly $|\mu|+|\lambda|$ terms, 
implying the vanishing \eqref{vanish}.

\subsection{}

It is well known that Jack polynomials labeled by single row or column have nonzero inner product with the 
power-sum function of the same degree (this is true already in the Schur function case $\theta=1$). 
Therefore, as we saw in Section \ref{s_single_product}, it suffices to prove \eqref{evalJack} 
in the special case 
$$
\lambda = (1^k)\,, \quad \mu = (l) \,.
$$
Since multiplication by a function adds squares to diagrams and there are only two diagram fitting 
inside both $\lambda$ and $\mu$, we have 
\begin{multline*}
 \left(E^m (E^*)^{\theta-m-1} J_\lambda, J_\mu \right)_\theta = \\ 
(E^m,J_\mu)_\theta \, (E^{\theta-m-1},J_\lambda)_\theta + \theta \, (E^m\,p_1,J_\mu)_\theta \, (E^{\theta-m-1}\,p_1,J_\lambda)_\theta \,, 
\end{multline*}
the factor of $\theta$ in the second term coming from $(p_1,p_1)_\theta^{-1} = \theta$. 

We have 
$$
(E^m,J_\mu)_\theta  = \theta^{-l} \prod_{i=0}^{l-1} (m-i)\,, \quad 
(E^m\, p_1,J_\mu)_\theta  = l \, \theta^{-l} \, \prod_{i=0}^{l-2} (m-i)\,,
$$
which is an elementary statement about polynomials in one variable. Dually, 
$$
(E^m,J_\lambda)_\theta  = \theta^{-k} \prod_{i=0}^{k-1} (m+i\theta)\,, \quad 
(E^m\, p_1,J_\lambda)_\theta  = l \, \theta^{-k} \, \prod_{i=0}^{k-2} (m+i\theta)\,. 
$$
Replacing $m$ by $\theta-1-m$ in the above formula and summing up gives the right answer. This concludes the proof.

\end{document}